\documentclass[12pt]{amsart}
\usepackage{amssymb,amsmath,amsthm,bm}
 \usepackage[colorlinks=true,allcolors=blue]{hyperref}
 
\usepackage[charter,cal=cmcal]{mathdesign}

\usepackage{color}

\def\e{{\rm e}}
\def\cic{\bm}
\def\eps{\varepsilon}

\def\d{{\rm d}}
\def\dist{{\rm dist}}

\def\R {\mathbb{R}}

\def\D {{\mathcal D}}

\def\psf{\mathsf{PSF}}

\def\T{{\mathsf{T}}}

\def\sign{{\mathrm{sign} }}

\def\Pp {{\mathbb P}}

\def \l {\langle}
\def \r {\rangle}

\def \and{\qquad\text{and}\qquad}

\newcommand{\supp}{\mathrm{supp}\,}

\newcounter{thms}
\newtheorem{proposition}{Proposition}
\newtheorem{theorem}[thms]{Theorem}
\newtheorem*{theorem*}{Theorem}
\newtheorem*{proposition*}{Proposition}
\newtheorem{corollary}[thms]{Corollary}
\newtheorem{lemma}[proposition]{Lemma}
\theoremstyle{definition}

\newtheorem{remark}[proposition]{Remark}

\numberwithin{proposition}{section}
\numberwithin{equation}{section}

\title[Domination of multilinear singular integrals]{Domination of multilinear singular integrals\\ by positive sparse forms}

 \author{Amalia Culiuc}
 \author{Francesco Di Plinio}
 \author{Yumeng Ou}
\address{\noindent Brown University Mathematics
Department, Box 1917,
Providence, RI 02912, USA}
\email[A.\ Culiuc]{amalia@math.brown.edu}
\email[F.\ Di Plinio]{fradipli@math.brown.edu}
\email[Y.\ Ou]{yumeng\_ou@brown.edu}

\thanks{FDP was partially
supported by the National Science Foundation under the grant
   NSF-DMS-1500449.}

\begin{document}
 \begin{abstract}
 We establish a uniform domination of the  family of trilinear multiplier forms with singularity over a one-dimensional subspace 
by  positive sparse  forms involving $L^p$-averages. This class includes the adjoint forms to the bilinear Hilbert transforms. Our result strengthens the $L^p$-boundedness proved by Muscalu, Tao and Thiele,  and entails as a corollary a novel rich multilinear weighted theory. A particular case of this theory is the  $L^{q_1}(v_1) \times L^{q_2}(v_2)$-boundedness of the bilinear Hilbert transform when the weights $v_j$ belong to the class $A_{\frac{q+1}{2}}\cap RH_2$. 
 Our proof relies on a stopping time construction based on newly developed localized outer-$L^p$ embedding theorems for the wave packet transform. {In  an Appendix, we show how our domination principle  can  be applied to recover the  vector-valued bounds for the bilinear Hilbert transforms recently proved by Benea and Muscalu.}\end{abstract}
\maketitle
\section{Introduction and main results}
The $L^p$-boundedness theory of Calder\'on-Zygmund  operators, whose prototype is the Hil\-bert transform, plays a central role in harmonic analysis and in its applications to elliptic partial differential equations, geometric measure theory and related fields. 

  A recent remarkable discovery is that the action of a singular integral operator $T$ on a function $f$ can be dominated in a pointwise sense by the averages of $f$ over a \emph{sparse}, i.e.\ essentially disjoint, collection of cubes in $\R^n$. This control is much stronger than  $L^p$-norm bounds and carries significantly more information on the operator itself. As of now, the most striking consequence is that    sharp weighted norm inequalities for $T$ follow from the corresponding, rather immediate estimates for the averaging operators.  Such a pointwise domination principle, albeit in a slightly weaker sense, appears explicitly for the first time in the proof of the $A_2$ theorem by Lerner \cite{Ler2013}. We also point out the recent improvements by Lacey  \cite{Lac2015} and Lerner \cite{Ler2015}, and the analogue for multilinear Calder\'on-Zygmund operators by Lerner and Nazarov \cite{LerNaz2015}. Most recently, Bernicot, Frey and Petermichl \cite{BFP} extend this approach to  non-integral singular operators associated with a second-order elliptic operator, lying outside  the scope of classical
Calder\'on-Zygmund theory.
    
   The   main focus of the present article is to formulate a similar principle for the class of multilinear multiplier operators, invariant under simultaneous modulations of the input functions, which includes the bilinear Hilbert transforms.  Besides their intrinsic interest, our results yield a rich, and sharp in a suitable sense, family of  multilinear weighted bounds for this class of operators. In fact,  Theorem \ref{MultApThm} below is the first result of this kind.
Weighted estimates for the bilinear Hilbert transforms have been mentioned as an open problem in several related works \cite{DoLac2012, DoLaceyWFS,GMe}. 
   
   %The $L^p$-boundedness of this family of bilinear singular integrals, conjectured by Calder\'on within his study of the Cauchy integral along Lipschitz curves, was first established in the celebrated series  \cite{LT1,LT2} by Lacey and Thiele.
 
 Let   $\Gamma=\{\xi=(\xi_1,\xi_2,\xi_3) \in \R^3: \xi_1+\xi_2+\xi_3=0\}$ and  $\beta \in \Gamma$ be a  fixed unit vector, nondegenerate in the sense that   $$
\Delta_{{\beta}}= \min_{k\neq j} |\beta_k -\beta_j| >0.
$$ We are concerned with the trilinear forms
 \begin{equation}
\label{IN1}
\Lambda_m(f_1,f_2,f_2) = \int_{\Gamma}  m(\xi) 
\prod_{j=1}^3 \widehat f_j(\xi_j) \, \d \xi 
\end{equation}
acting on triples of Schwartz functions on $\R$, where $m:\Gamma \to \mathbb C$ is a Fourier multiplier satisfying, in multi-index notation, 
\begin{equation}
\label{decay}
\sup_{|\alpha| \leq N}\sup_{ \xi \in\Gamma}\big(  \dist(  \xi,  \beta ^\perp)\big)^\alpha \big| \partial_\alpha m ({\xi})\big| \leq C_N.
\end{equation}
 The one-parameter family (with respect to $\beta$) of trilinear forms adjoint to the bilinear Hilbert transforms is obtained by choosing
$$
m(\xi)=\sign(\xi \cdot \beta).
$$
 In \cite{MTT}, substantially elaborating on the seminal work by Lacey and Thiele \cite{LT1,LT2}, Muscalu, Tao and Thiele prove the following result. 
\begin{theorem}\cite[Theorem 1.1]{MTT} \label{ThmMTT} Let $m$ be a multiplier satisfying \eqref{decay}. Then
the adjoint bilinear operators $T_m$ to the  forms $\Lambda_m$ of \eqref{IN1} have the mapping properties
\begin{equation}
\label{IN2}
T_m: L^{q_1}(\R) \times L^{q_2}(\R) \to L^{\frac{q_1q_2}{q_1+q_2}}(\R)
\end{equation}
for all exponent pairs $(q_1,q_2)$ satisfying   $1<\inf\{q_1,q_2\} <\infty$ and
\begin{equation}
\label{IN3} \textstyle \frac{1}{q_1}+\frac{1}{q_2}< \frac{3}{2}.
\end{equation}
\end{theorem}
Not unexpectedly,   a pointwise domination principle for this class of bilinear operators is not allowed to hold, as   we elaborate in Remark \ref{noucrem} below.
This obstruction is overcome  by introducing   the closely related notion of \emph{domination by sparse positive  forms} of the adjoint trilinear form, which we turn to in what follows.

We say that $\mathcal S$ is a $\eta$-\emph{sparse collection} of intervals $I\subset \R$ if   for every $I\in \mathcal S$ there exists a measurable $E_I\subset I$ with $|E_I|\geq \eta |I|$ such that $\{E_I: I \in \mathcal S\}$ are pairwise disjoint.  The positive sparse trilinear form of type $\vec{p}=(p_1,p_2,p_3)$ associated to the sparse collection $\mathcal S$ is defined by
\begin{equation}
\label{posop1}
\psf_\mathcal{S}^{\vec p} (f_1,f_2,f_3) (x) = \sum_{I\in\mathcal{S}}|I|\prod_{j=1}^3\l f_j \r_{I,p_j}, \qquad \l f \r_{I,p}:= \left( \frac{1}{|I|}\int_I |f(x)|^p\, \d x\right)^{\frac1p};
\end{equation}
we omit the subscript and write $\l f \r_{I}$ when $p=1$.
A rather immediate consequence of the   Hardy-Littlewood maximal theorem is the following proposition.\footnote{ We omit the proof, which is a simplified version of the proof of Corollary \ref{CorVV} given in the appendix}
\begin{proposition} \label{uptype}Let $T$ be a bilinear operator. Suppose that for all tuples $ (f_1,f_2,f_3)\in \mathcal C_0^\infty(\R)^3$ there holds 
\[
|\l T(f_1, f_2),f_3\r| \leq K\sup_{\mathcal S\,\eta\mathrm{-sparse}}  \psf_\mathcal{S}^{\vec p} (f_1,f_2,f_3)
\]
Then  for all  $(f_1,f_2)\in \mathcal C_0^\infty(\R)^{2}$ there holds
 \begin{equation}
\label{LpSparse}
\| T(f_1,f_2)\|_{ \frac{q_1q_2}{q_1+q_2} } \leq K C_{q_1, q_{2},\eta}  \prod_{j=1}^{2} \|f_j\|_{q_j} \end{equation}
provided that $p_j<q_j\leq \infty$ for  $j=1,2$ and $\inf\{q_1,q_2\}<\infty$.
\end{proposition}
 Our main result is  a strengthening of Theorem \ref{ThmMTT} to a domination by positive sparse forms. To formulate it, we need one more notion. We say that $\vec p=(p_1,p_2,p_3)$ is an \emph{admissible tuple}  if
\begin{equation}
\label{admtuple}
1\leq p_1,p_2,p_3 <\infty,    \qquad \eps(\vec p) :=2 -\sum_{j=1}^3 \textstyle   \frac{1}{\min\{p_j,2\}} \geq 0
\end{equation}
If all the constraints hold with strict inequality, we say that $\vec{p}$ is an \emph{open admissible tuple}. 
 \begin{theorem} \label{ThmMain}    Let $\vec p$ be an open admissible tuple. There exists $K=K(\vec p), N=N(\vec p)$ such that the following holds.
For any tuple $ (f_1,f_2,f_3)\in \mathcal C_0^\infty(\R)^3$  there exists a $\frac16$-sparse collection $\mathcal S$ such that
\begin{equation}
\label{IN5}
\sup_{m}|\Lambda_m(f_1,f_2,f_3)|\leq K C_N  \psf_\mathcal{S}^{\vec p}(f_1,f_2,f_3),
\end{equation}
where the supremum is being taken over the family of multipliers $m$ satisfying \eqref{decay}.
\end{theorem}
We stress that the constants $K$ and $N$ depend only on the exponent tuple $\vec p$, and the choice of the sparse collection $\mathcal S$ depends only on $f_1,f_2, f_3$ and $\vec p$ and is, in particular, independent of the multiplier $m$. 
\begin{remark}[Sharpness of Theorem \ref{ThmMain}] \label{remsharp} Let $(q_1,q_2)$ be an exponent pair with $1<\inf\{q_1,q_2\}<\infty$.  Then there exists an open admissible tuple $\vec p=(p_1,p_2,p_3)$ with $p_1<q_1,p_2<q_2$ if and only if \eqref{IN3} holds for $(q_1,q_2)$. This observation, coupled with Proposition \ref{uptype}, yields Theorem \ref{ThmMTT} as a corollary of Theorem \ref{ThmMain}. 

On the other hand, let $\phi$ be an even Schwartz function with
$
\cic{1}_{[-2^{-4}, 2^{4}]} \leq  \widehat \phi \leq  \cic{1}_{[-2^{-3}, 2^{-3}]},
$ $\{\beta,\gamma\}$ be an orthonormal basis of $\Gamma$.  Define the family of multipliers on $\Gamma$
\begin{equation}
\label{counterex}
m_{\vec \sigma, M} (\xi) = \sum_{n=0}^{M-1} \sigma_{n} \widehat \phi\left(2^{8}(\xi_1-(\eta^{n})_1) \right)\widehat \phi\left(2^{8}(\xi_2- (\eta^{n})_2 )\right)  \widehat \phi\left(\xi_3-(\eta^{n})_3 \right)
\end{equation}
where $\eta^{n}= n \gamma + \beta$, $n\in \mathbb N$.
The same argument as in  \cite[Section 2.2]{Lac} yields
\[
\sup_{\vec \sigma \in \{-1,1\}^M} \left\|T_{m_{\vec \sigma, M}}\right\|_{L^{q_1}\times L^{q_2} \to L^{\frac{q_1q_2}{q_1+q_2}}} \geq C M^{\frac{1}{q_1}+ \frac{1}{q_2} -\frac32}
\]
while the family $\{m_{\vec \sigma, M}: M\in \mathbb N,\vec \sigma \in \{-1,1\}^M \}$ satisfies \eqref{decay} uniformly. This implies that the range \eqref{IN3} of Theorem \ref{ThmMTT} is sharp up to equality holding in \eqref{IN3} and, in turn, that \eqref{IN5} cannot hold for any tuple violating \eqref{admtuple}. Hence, Theorem \ref{ThmMain} is sharp up to possibly replacing the assumption \emph{open admissible} with the stronger \emph{admissible}. The behavior of the forms $\Lambda_m$ for tuples at the boundary of the admissible region is studied in detail in \cite{DPTh2014}.
\end{remark}
\begin{remark}[No uniform control by a bilinear positive sparse operator] \label{noucrem} For bilinear Calder\'on-Zygmund operators $T$, there holds a pointwise domination by sparse operators of the type
\[
|T(f_1,f_2)(x)| \leq C \sum_{I\in\mathcal{S}(f_1,f_2)}   \l f_1 \r_{I,p_1}\l f_2 \r_{I,p_2} \cic{1}_I(x).
\]
One can take $p_1=p_2=1$: see \cite{LerNaz2015}. Essentially self-adjoint operators $T$ enjoying such pointwise domination inherit the  boundedness property
$$ T: L^{1} \times L^{p_j} \to L^{\frac{p_j}{1+p_j},\infty}$$  
which, as described in  the previous Remark \ref{remsharp},  fails for the generic  $T_m$ of the class \eqref{decay} when $\inf\{p_1,p_2\}<2$. In fact, no $L^1$-boundedness properties are expected to hold even for the bilinear Hilbert transforms. Summarizing,  no such pointwise domination principle can be obtained for   $T_m$ when $\inf\{p_1,p_2\}<2$ and, most likely, neither for the case when $\inf\{p_1,p_2\}\geq 2$.  Our formulation in terms of  positive sparse forms overcomes this obstacle: a similar idea, albeit not explicit, appears in the linear setting in \cite{BFP}. 
\end{remark}
  Theorem \ref{ThmMain} implies multilinear  weighted bounds for the forms $\Lambda_m$. Our  main weighted theorem will involve  multilinear $A_{\vec q}^{\,\vec p}$ Muckenhoupt constants. Given any tuple $\vec p$, a H\"older tuple $\vec{q}$ and 
a weight vector $\vec v=(v_1,v_2,v_3)$ 
satisfying 
\begin{equation} \label{1weight}
 \prod_{j=1}^3 v_j^{\frac{1}{q_j}}=1,
\end{equation}
these are defined as
\begin{equation} \label{Apq}
    [\vec v]_{A_{\vec q}^{\,\vec p}}:=\sup_{I \subset \R   } \prod_{j=1}^3\big\langle v_j^{\frac{p_j}{p_j-q_j}} \big\rangle_I^{\frac{1}{p_j}-\frac{1}{q_j}}.
\end{equation}
For $\vec{p}=(1,1,1)$, these weight classes have been introduced in \cite{LerOmbPer}, to which we send for an exhaustive discussion of their properties. A particular case of  \eqref{Apq} (where $p_1=1$) can be found in \cite{Kabe} as a necessary and sufficient condition for weighted $L^q$-boundedness of   the bilinear fractional  integrals. Furthermore, the classes \eqref{Apq} appear in ongoing work on multilinear Calder\'on-Zygmund operators satisfying H\"ormander type conditions \cite{CTW}. 
\begin{theorem} \label{MultApThm}
Let $\vec{q} $ be a H\"older tuple with $1<q_1,q_2,q_3<\infty$ and $\vec v$ be a weight vector  satisfying \eqref{1weight}. Then
 there holds
\[ \displaystyle
 \sup_{m } \left|\Lambda_m(f_1,f_2,f_3)\right| \leq  
\left(\inf_{\vec p}  C(\vec p,\vec q ) [\vec v]_{A_{\vec q}^{\,\vec p}}^{ \max\left\{\frac{q_j}{q_j-p_j}\right\}} \right) \prod_{j=1}^3 \|f_j\|_{L^{q_j}(v_j)} 
 \]
where the supremum is being taken over the family of multipliers $m$ satisfying \eqref{decay},  the infimum is taken over open admissible tuples $\vec p $ with $p_j<q_j$, and
\begin{equation}
\label{Cpq}
C(\vec p,\vec q ) =  K(\vec p ) C_{N(\vec p)}\left(  \prod_{j=1}^{3} {\textstyle \frac{q_j}{q_j-p_j}} \right)2^{3\left(\sum_{j=1}^3 \frac{1}{p_j}-1\right) \max\left\{\frac{p_j}{q_j-p_j}\right\}}.
\end{equation}
\end{theorem}

One is usually interested in weighted estimates involving Muckenhoupt and reverse H\"older constants of each single weight. Recall that the $A_q$ and $RH_{\alpha}$ constant of a weight $v$ on $\R$ are defined as
\[
[v]_{A_q}:= \sup_{I \subset \R  } \langle v  \rangle_I  \langle v^{\frac{1}{1-q}}  \rangle_I^{q-1}, \qquad [v]_{RH_\alpha}:= \sup_{I \subset \R } \langle v^{\alpha}  \rangle_I^{\frac1\alpha}    \langle v  \rangle_I^{-1},  
\] A  suitable choice of admissible tuple $\vec p$  in Theorem \ref{MultApThm} yields the following corollary.\footnote{We have come to know  that Xiaochun Li \cite{LiPC} has some unpublished results about weighted estimates for the bilinear Hilbert transforms.}
\begin{corollary} \label{Aqcor}
Let 
$$ \textstyle
1<q_1, q_2, r=\frac{q_1q_2}{q_1+q_2}<\infty.
$$
and $v_1,v_2$ be given weights with $v_1^2\in  A_{q_1}, v_2^2  \in A_{q_2}$. Then  the operator norms
$$
T_{m}:L^{q_1}(v_1)\times L^{q_2}(v_2)  \to L^{r}(u_3), \qquad u_{3}:= \prod_{j=1}^2 v_j^{\frac{r}{q_j}}
$$
 of the family            of multipliers  satisfying \eqref{decay} with uniform constants $C_N$ are uniformly bounded above by a positive constant depending on $\{q_j,\, [v_j^2]_{A_{q_j}}:j=1,2\}$ only.\end{corollary}

 We refer to the recent monograph \cite{CruzMartellPerez}  for details on the $A_q$ and $RH_{\alpha}$ classes. Here we remark that if $q>1$ then \cite[Section 3.8]{CruzMartellPerez}
$
v   \in A_{\frac{q+1}{2}} \cap RH_2
$
if and only if 
$
v^2 \in A_{q}.
$
 We mention that a theory of \emph{linear} extrapolation for weights in the  $A_{q} \cap RH_\alpha $ classes has been introduced in \cite{AusMarI}; see also the already mentioned monograph \cite{CruzMartellPerez}.

 As a further application of Corollary \ref{Aqcor}, weighted, vector-valued estimates for multipliers $T_m$ satisfying condition \eqref{decay}, extending the results of \cite{BenMusc15,Silva12}
can be obtained by a multilinear version of the extrapolation theory of \cite{AusMarI}. These extensions are the object of an upcoming companion article by the same authors. 
{ However, Theorem \ref{ThmMain} can  be employed to recover the unweighted vector-valued estimates of 
\cite{BenMusc15,Silva12} in a rather direct fashion.  In order to keep our outline  as simple as possible, we postpone the complete statement and proof of the vector-valued estimates to Appendix \ref{SecVV}.}

\subsection*{Structure of the article and proof techniques} 
The class of multipliers \eqref{decay}, in addition to the familiar invariances under isotropic dilations and translations proper of Coifman-Meyer type multipliers, enjoys  a one-parameter invariance under simultaneous modulation of the three input functions along the line $\R\gamma=\{\beta,(1,1,1)\}^\perp$.  The invariance properties of the class \eqref{decay} are essentially shared by a family of discretized   trilinear forms involving the maximal wave packet coefficients of the input functions parametrized by  rank 1 collection of tritiles, which we call \emph{tritile form}.

The first step in the proof of Theorem \ref{ThmMain}, carried out in Section \ref{tritilemapsec}, is to establish that for any multiplier $m$ satisfying \eqref{decay}, the form $\Lambda_m$ lies in the convex hull of finitely many tritile forms.  This discretization procedure is largely the same as the one employed in \cite{MTT}.  Theorem \ref{ThmMain} then reduces to the analogous result for   tritile forms, Theorem \ref{ThmDisc}. It is of paramount importance here that the sparse collection $\mathcal S$ constructed in Theorem \ref{ThmDisc} is independent of the particular tritile form.

The explicit construction of the collection $\mathcal S$, and in fact the proof of Theorem \ref{ThmDisc}, is performed in Section \ref{Secpf} by means of an inductive argument. The intervals of $\mathcal S$ are, roughly speaking, the stopping intervals of the $p_j$-Hardy-Littlewood maximal function of the $j$-th input. At each stage of the argument,   the contribution of those wave packets localized within one of the stopping intervals will be estimated at the next step of the induction, after a careful removal of the tail terms. The main term, which is the contribution of the wave packets whose spatial localization is not contained in the union of the stopping interval is estimated by means of a localized  outer $L^{p_j}$ embedding Theorem for the wave packet transform. 

 This outer $L^p$ embedding, which is the concern of Proposition \ref{CET}, is a close relative of the main result of \cite{DPOu2} by two of us, namely, a localized embedding theorem for the \emph{continuous} wave packet transform. In fact,   while   Proposition \ref{CET} is proved here  via a transference argument based upon \cite[Theorem 1]{DPOu2}, a direct proof can be given by repeating the arguments of \cite{DPOu2} in the discrete setting. The construction of the  outer $L^p$ spaces on rank 1 collections, which parallels the outer $L^p$ theory introduced by Do and Thiele in \cite{DoThiele15},  is performed in Section \ref{SecOut}.  

% Finally,
 Section \ref{SecWeight} contains the proof of the weighted estimates of  Theorem \ref{MultApThm} and \ref{Aqcor}, 
{and the concluding Section \ref{SecVV} is dedicated to   vector-valued extensions.}

%
%\begin{remark} { TBC} Condition  \eqref{decay} is the modulation-invariant  analogue of the one satisfied by the    Coifman-Meyer   multilinear multipliers
%\begin{equation}
%\label{decayCM}
%\sup_{|\alpha| \leq N}\sup_{ \xi \in\Gamma}|\xi|^\alpha \big| \partial_\alpha m ({\xi})\big| \leq C_N
%\end{equation}
%also known as multilinear  (translation-invariant) Calder\'on-Zygmund singular integral operators. We send to the monograph \cite{MuscSchlII} for a comprehensive treatment of both classes \eqref{decay} and \eqref{decayCM}. 
%\end{remark}

\subsection*{Notation} 
Let   $\chi(x)= (1+|x|^2)^{-1}$. For an interval $I$ centered at $c(I)$ and of length $\ell(I)=|I|$, we write
\begin{equation}
\label{chiI}
\chi_I (x) :=\textstyle \chi\left( \frac{x-c(I))}{\ell(I)} \right).
\end{equation}
We will  make use of the weighted $L^p$ spaces
\[
\|f\|_{L^p(\chi_I^N)}:= \left(\frac{1}{|I|}\int_\R |f(x)|^p (\chi_{I}(x))^N \right)^{1/p}, \; 1\leq p <\infty, \qquad   \|f\|_{L^\infty(\chi_I^N)} = \|f\chi_{I}^N\|_\infty. 
\] with $N$ positive integer. 
We write 
\[
\mathrm{M}_{p}(f)(x) =\sup_{I \subset \R} \l f \r_{I,p} \cic{1}_I(x)
\]
for the $p$-Hardy Littlewood maximal functions. Finally, the constants implied by almost inequality sign $\lesssim$ and the comparability sign $\sim$  are meant to be absolute throughout the article.

\subsection*{Acknowledgments}  The authors want to thank   David Cruz-Uribe, Kabe Moen and Ro\-dol\-fo Torres for providing   additional insight on multilinear weighted theory. The authors are grateful to  Gennady Uraltsev for fruitful discussions on the notion of localized outer $L^p$ embeddings.  
\section{Tritile maps} \label{tritilemapsec}
In this section, we reduce Theorem \ref{ThmMain} to the corresponding statement for a class of multilinear forms which we call \emph{tritile maps}. Throughout, we assume that the nondegenerate unit vector $\beta \in \Gamma$ is fixed and let $\gamma\in \Gamma$ be a unit vector perpendicular to $\beta$, spanning the singular line of the multipliers $m $ from \eqref{decay}. 

\subsection{Rank 1 collection of tri-tiles}

 A \emph{tile} $T=I_T \times \omega_T$ is the cartesian product of two intervals $I_T,\omega_T$ with $|I_T||\omega_T|\sim 1$. A tri-tile
 $P=(P_1,P_2,P_3)$ is an ordered triple of tiles $P_j, j=1,2,3$ with the property that 
 \[
 I_{P_1}=I_{P_2}=I_{P_3}=:I_{P};
 \]
we denote by $\vec{\omega_P}= \omega_{P_1}\times \omega_{P_2}\times \omega_{P_3} $ the \emph{frequency cube} corresponding to $P$ and by $\omega_{P}$ the convex hull of the intervals $3\omega_{P_j},$ $j=1,2,3.$
We say that the collection of tri-tiles $\mathbb{P}$ is of rank 1 if
\begin{itemize}
\item[a.] $\mathcal{I}=\{I_P:P \in \Pp\}$ and  $\Omega_j=\{ \omega_{P}: P \in \Pp \}$, $j=1,2,3$ are  $\log \mathsf{g}$  scale-separated dyadic grids;
\item[b.] if $P\neq P' \in \Pp$ are such that  $I_P=I_{P'}$ then  $\omega_{P_j}\cap\omega_{P'_j}=\emptyset$ for each  $j \in\{1,2,3\}$;

\item[c.] if $P,Q\in \Pp$ are such that $\omega_{P_j} \subset  \omega_{Q_j}$ for some $j\in \{1,2,3\}$ then $\mathsf{g} \omega_{P} \subset \mathsf{g} \omega_{Q}$;

\item[d.] if $P,Q\in \Pp$ are such that $\omega_{P_j} \subset  \omega_{Q_j}$ for some $j\in \{1,2,3\}$ then $3\omega_{P_k} \cap  3 \omega_{Q_k}=\emptyset$ for $k\in \{1,2,3\}\setminus\{j\}$. 
\end{itemize}
We can take $\mathsf{g}\sim (\Delta_{\cic \beta})^{-1}$. 

\subsection{Tritile forms} Let $A_N$ be a fixed increasing sequence of positive constants. For each  tile $T$ we define the \emph{adapted family} $\cic{\Phi}(T)$ to be the collection of Schwartz functions $\phi_T$ satisfying 
{ {\begin{equation}
\label{adapt}
\sup_{n\leq N} \sup_{x\in \R} {|I_{T}|^{n+1}}{ \chi_I(x)^{-N}} \left| \left( \mathrm{e}^{-ic(\omega_T)\cdot} \phi_T(\cdot)\right)(x)\right| \leq A_N, \qquad \textrm{supp}\,\widehat{\phi_T}\subset \omega_{T}.
\end{equation}}}
Let $\Pp$ be a rank 1 collection of tritiles and $f_j\in L^1_{\mathrm{loc}}(\mathbb R)$. We define the tritile maps $F_j:\mathbb P\to \mathbb C$ by
\begin{equation}
\label{tritilemaps}
F_j(f)(P)=\sup_{\phi_{P_j} \in \cic{\Phi}(P_j) } |\l f_j, \phi_{P_j}  \r|, \qquad j=1,2,3
\end{equation}
and the trisublinear tritile form associated to $\Pp$ by
\begin{equation}
\label{tritileform}
\Lambda_{\Pp}(f_1,f_2,f_3)= \sum_{P \in \Pp} |I_{P}| \prod_{j=1}^3 F_j(f_j)(P).
\end{equation}

\subsection{Reduction to uniform bounds for tritile forms} The following lemma  is a reformulation of the well-known discretization procedure from \cite{MTT}. Several versions of this procedure have since appeared, see for instance the monographs \cite{MuscSchlII,ThWp}. We omit the standard (by now) proof. 
\begin{lemma} \label{modellemma} There exists a finite collection $\{\mathbb P^{1},\ldots, \Pp^{\mathsf G}\}$ of rank 1 collections of tritiles such that, for any multiplier $m$ satisfying \eqref{decay} and any tuple of Schwartz functions $f_1,f_2,f_3$, there holds 
\[
|\Lambda_{m}(f_1,f_2,f_3)| \leq  \sum_{j=1}^\mathsf{G}
\Lambda_{\Pp^j}(f_1,f_2,f_3) \]
and the adaptation constants $\{A_N\}$ of the adapted families defining $\Lambda_{\Pp^j}$ depend on $\{C_{N}\}$ only.
Furthermore, the character $\mathsf{G}$ depends only on the nondegeneracy constant of $\beta$.
\end{lemma}
Theorem \ref{ThmMain} is then an immediate consequence of Lemma \ref{modellemma} and of the following discretized version, whose proof is given in Section \ref{Secpf}. \begin{theorem} \label{ThmDisc}    Let $\vec p$ be an open admissible tuple. There exists $K=K(\vec p), N=N(\vec p)$ such that the following holds.
For any tuple $ (f_1,f_2,f_3)$ with $f_j\in L^{p_j}(\R)$ and compactly supported   there exists a $\frac16$-sparse collection $\mathcal S$  such that
\[
\sup_{\Pp}\Lambda_\Pp(f_1,f_2,f_3)\leq K A_N  \psf_\mathcal{S}^{\vec p}(f_1,f_2,f_3),
\]
where the supremum is being taken over all rank 1 collections of tritiles $\Pp$ of finite cardinality and adaptation sequence $\{A_N\}$. In particular, the collection $\mathcal S$ depends only on $f_1,f_2,f_3$ and the tuple $\vec p$.
\end{theorem}

\section{Outer $L^p$ spaces of tritiles} \label{SecOut}
In this section, we formulate the outer measure space that is needed for our proof, which is based on a finite rank 1 collection of tritiles $\Pp$. Recall that $\omega_P=\mathrm{co}(3\omega_{P_1},3\omega_{P_2},3\omega_{P_3} )$. The generating collection is the set of trees $\mathcal T\subset \mathcal P(\Pp)$. The set $\T\subset \Pp$ is a tree with top data $(I_\T,\xi_\T)$ if
\[
I_P \subset I_\T, \qquad \xi_T \in \omega_P \qquad \forall P \in \Pp.
\] 
By property d.\ of the rank 1 collections, we have that each tree $\T$ can be written as the union 
\begin{equation}\label{Treesplit}
\T=\bigcup_{1\leq j<k\leq 3} \T\setminus (\T_j\cup \T_k)
\end{equation}
where each $\T_{j}$ is a tree with the same top data as $\T$ and has the additional property
\[
 \{3\omega_{P_k}:P \in \T_j\} \textrm{ are a pairwise disjoint collection for } k \neq j.
\]
The premeasure $\sigma:\mathcal{T}\to[0,\infty)$ is given by
\[
\sigma (\mathsf{T}):=|I_\mathsf{T}|.
\]
We now define a  tuple of  \emph{sizes} on $\Pp$, that is, homogeneous and quasi-subadditive maps $\mathbb C^{\Pp} \to [0,\infty]^{\mathcal T}$. For each $j=1,2,3$ we define the  corresponding \emph{size} on  functions $F: \Pp \to \mathbb C$ by
\begin{equation}
\label{sizedisc}
\mathsf{s}_j(F) (\mathsf{T}):=\left( \frac{1}{|I_\mathsf{T}|}\sum_{P \in \T\setminus \mathsf{T}_{j}} |I_P| |F(P)|^2 \right)^{\frac12} + \sup_{P \in  \mathsf{T}} |F(P)|,
\end{equation} and denote the corresponding outer measure spaces as $(\Pp, \sigma,\mathsf s_j)$ and outer $L^p$ spaces as
$$
L^p(\Pp, \sigma,\mathsf s_j), \qquad 1\leq p \leq \infty.
$$

Here we recall that for $f\in\mathcal{B}(\Pp)$,
\[
\|f\|_{L^\infty(\Pp,\sigma,\mathsf{s}_j)}:=\sup_{\T\in\mathcal{T}}\mathsf{s}_j(f)(\T),
\]
\[
\|f\|_{L^p(\Pp,\sigma,\mathsf{s}_j)}:=\left(\int_0^\infty p\lambda^{p-1}\mu(\mathsf{s}_j(f)>\lambda)\,\d\lambda\right)^{1/p},\quad 0<p<\infty,
\]
where the super level measure $\mu(\mathsf{s}_j(f)>\lambda)$ is defined to be the infimum of all values $\mu(E)$ ($\mu$ being the outer measure generated by the premeasure $\sigma$), for $E$ running through all Borel subset of $\Pp$ such that 
\[
\sup_{\T\in\mathcal{T}}\mathsf{s}_j(f\cic{1}_{E^c})(\T)\leq\lambda.
\]

We also note that there holds the following H\"older's inequality:
\begin{lemma}\label{Holder}
Let
\begin{equation}
\vec{q}=(q_1,q_2,q_3),\qquad 1\leq q_j\leq\infty,\qquad \sum_{j=1}^3\frac{1}{q_j}=1
\end{equation}
be a H\"older tuple. Let $G_j:\Pp\rightarrow \mathbb{C}$, $j=1,2,3$. Then
\[
\sum_{P\in\Pp}|I_P|\prod_{j=1}^3|G_j(P)|\lesssim \prod_{j=1}^3\|G_j\|_{L^{q_j}(\Pp,\sigma,\mathsf{s}_j)}
\]
with absolute implicit constant.
\end{lemma}
\begin{proof}
Define another size
\[
\mathsf{s}^1(F)(\T):=\frac{1}{|I_\T|}\sum_{P\in\T}|I_P||F(P)|.
\]Then it is obvious that for any $\T$ there holds
\[
\sum_{P\in\T}|I_P|\prod_{j=1}^3|G_j(P)|\leq \sigma(\T)\mathsf{s}^1\left(\prod_{j=1}^3 G_j\right)(\T),
\]which by the Radon Nikodym proposition in \cite{DoThiele15} implies that
\[
\sum_{P\in\Pp}|I_P|\prod_{j=1}^3|G_j(P)|\lesssim \Big\|\prod_{j=1}^3 G_j\Big\|_{L^1(\Pp,\sigma,\mathsf{s}^1)}.
\]
Furthermore, according to (\ref{Treesplit}) and the classical H\"older's inequality, one can easily check that for any fixed $\T$,
\[
\begin{split}
&\mathsf{s}^1\left(\prod_{j=1}^3 G_j\right)(\T)\\
\leq& \frac{1}{|I_\T|} \sum_{1\leq j<k\leq 3}\sum_{P\in\T\setminus\left(\T_{j}\cup \T_{k}\right)}|I_P||\prod_{j=1}^3 G_j(P)|\\
\leq&\sum_{1\leq j<k\leq 3}\left[\left(\sum_{P\in\T\setminus \T_{j}}|I_P||G_{j}(P)|^2\right)^{1/2}\left(\sum_{P\in\T\setminus \T_{k}}|I_P||G_{k}(P)|^2\right)^{1/2}\prod_{i\neq j,k}\sup_{P\in\T}|G_i(P)|\right]\\
\lesssim& \prod_{j=1}^3 \mathsf{s}_j(G_j)(\T).
\end{split}
\]Hence the outer H\"older inequality in \cite{DoThiele15} yields that
\[
\Big\|\prod_{j=1}^3 G_j\Big\|_{L^1(\Pp,\sigma,\mathsf{s}^1)}\lesssim \prod_{j=1}^3\|G_j\|_{L^{q_j}(\Pp,\sigma,\mathsf{s}_j)},
\]which completes the proof.
\end{proof}

\section{Localized Carleson embeddings} \label{SecEmb}
In this section, when we write \emph{dyadic interval}, we mean intervals $I\in \mathcal D$, where $\mathcal D$ is a  fixed dyadic grid on $\R$.
Fix a dyadic interval $Q\subset \R$ and   $f\in L^p (\R)$ with $\mathrm{supp}\, f \subset 3Q$.  We define the $p$-stopping intervals of $f$ on $Q$ by 
\begin{equation}  \label{maxfctbd}  
{\mathcal I}_{f,p,Q}= \textrm{{maximal\ dyadic}}\, I \subset Q\, \textrm{s.t.\ } I\subset\left\{x\in \R:\mathrm{M}_{p} f(x) > \mathsf{C} \l f\r_{3Q,p}\right\}
  \end{equation}
Notice that ${\mathcal I}_{f,p,Q}$ is a pairwise disjoint collection of dyadic  intervals   and that the maximal theorem guarantees the sparseness condition
\begin{equation}
\label{sparsecond}
\sum_{I\in {\mathcal I}_{f,p,Q}} |I| \leq\left| \left\{x\in \R:\mathrm{M}_{p} f(x) > \mathsf{C} \l f\r_{3Q,p}\right\}\right| \leq  { 
\frac{|Q|}{6}}\end{equation}
provided $\mathsf{C}$ is chosen large enough. Furthermore, from the very definition of ${\mathcal I}_{f,p,Q}$, there holds
\begin{equation}
\label{doublingcond}
\inf_{x\in 3I}\mathrm{M}_{p}f(x)\lesssim \l f\r_{3Q,p} \qquad \forall I \in {\mathcal I}_{f,p,Q}.
\end{equation}

In what follows, we fix a finite collection of rank 1 tritiles $\Pp$ whose intervals $\{I_P:P\in \Pp\}$ are dyadic. We introduce the notation $\Pp_\leq(I)=\{P \in \Pp: I_P \subset I\}$ and the set of \emph{good tritiles}
\begin{equation}  \label{excset1}  
\mathsf{G}_{f,p,Q}=\Pp \setminus \left( \bigcup_{I\in {\mathcal I}_{f,p,Q}}  { \Pp_\leq}(I)\right). 
 \end{equation}
Recalling the definition of the tritile maps from \eqref{tritilemaps}, we have the following proposition, which is used to control the main term of the tritile forms \eqref{tritileform} localized to $ 3Q $.
 \begin{proposition} \label{CET}Let  $Q\subset \R$ be a dyadic interval and $f$ be a Schwartz function. For any $1<p<2$, $q>p'$  there exists $N=N(p,q)$ and $K=K(p,q)$ such that
 \begin{equation} \label{cet} \tag{LC$_{q,p}$}  \left\|F_j (f\cic{1}_{3Q}) \cic{1}_{\mathsf{G}_{f,p,Q}}\right\|_{L^{q}(\Pp, \sigma,\mathsf s)} \leq KA_N  |Q|^{\frac1q}   \l f\r_{3Q,p}.
\end{equation} 
 \end{proposition}
 \subsection{Proof of  Proposition \ref{CET}} The Proposition will be proved by a transference argument  using the main result of \cite{DPOu2}, which is the continuous parameter version recalled below. However, Proposition \ref{CET} may also be obtained directly, by repeating the arguments of \cite{DPOu2} in the (simpler, in fact) discrete parameter setting. We leave the details to the interested reader.
  
 \subsubsection{A continuous parameters version of Proposition \ref{CET}}We need to define the continuous outer measure space on the base set   
 \[{\Pp^\circ}=[-R,R]\times(0,R] \times [-R,R], \qquad 
 R=10\max\{c(I_P)+|I_P| +c(\omega_P)+|\omega_P|:P \in \Pp\};   
 \]
 we are using that $\Pp$ is a finite set.
Let $I\subset \R$ be an interval and $\xi \in \R$.  The corresponding \emph{generalized  tent} and its \emph{lacunary part},  with fixed geometric parameters $\mathsf{g}, \mathsf{b}$,  are defined by 
\[
\begin{split}&
\T^\circ(I,\xi)= \{(u,t,\eta) \in {\Pp^\circ}:0<t<|I|, |u-c(I)| <|I|-t, |\eta-\xi| \leq \mathsf{g}t^{-1}\}, \\ & \T^\circ_\ell(I,\xi)= \{(u,t,\eta) \in \T^\circ(I,\xi): t|\xi-\eta|>\mathsf b\}.
\end{split}
\]
We use the superscript $^\circ$ to distinguish discrete \emph{trees} $\T$ with top data $(I_\T,\xi_\T)$ from continuous \emph{tents}   $\T^\circ$ with same top data $(I_\T,\xi_\T)$.
It will also be convenient to use the notation
\[
\T^\circ(I )= \{(u,t,\eta) \in {\Pp^\circ}:0<t<|I|, |u-c(I)| <|I|-t,  \} \]
for the projection of $\T^\circ(I,\xi)$ on the first two components.

An outer measure $\mu^\circ$ on    $   {\Pp^\circ}$, with \[\mathcal T^\circ=\big\{\T^\circ(I,\xi):I \subset [-R,R], \xi \in [-R,R]\big\}\] as generating collection is then defined via the premeasure $\sigma(\T^\circ(I,\xi))=|I|$. 
For $F: Z\to \mathbb C$ Borel measurable,  we define the   size
\begin{equation}
\label{sizecirc} 
\mathsf{s}^\circ(F) (\T^\circ(I,\xi)):= \left(\frac{1}{|I|} \int_{\T_\ell^\circ(I,\xi)} |F(u,t,\eta)|^2 \,  \d u \d t\d \eta  \right)^{\frac12} +\sup_{(u,t,\eta) \in \T^\circ(I,\xi)} |F(u,t,\eta)|.
\end{equation} 
Denoting by 
$
L^p(  {\Pp^\circ}, \sigma , \mathsf{s}^\circ), \; L^{p,\infty}({\Pp^\circ}, \sigma, \mathsf{s}^\circ)
$
the corresponding strong and weak outer $L^p$ spaces, we turn to the reformulation of  the main result of \cite{DPOu2}.  A family  of Schwartz functions 
$$
{\Phi}:=\{\phi_{u,t,\eta}: (u,t,\eta) \in {\Pp^\circ} \}  
$$
is said to be an \emph{adapted system}  with \emph{adaptation constants} $A_{N}$ if
\begin{equation}
\label{adaptxi}
\sup_{(u,t,\eta)\in {\Pp^\circ}}\sup_{\substack{n\leq N }}  \sup_{x\in \R}  t^{n+1} \chi\left(\frac{x-u}{t}\right)^{-N} \left|   (\e^{-i\eta \cdot} \phi_{u,t,\eta}(\cdot))^{(n)}(x) \right| \leq A_{N}
\end{equation}
for all nonnegative integers $N$ and furthermore \[
t|\zeta-\eta|>1\implies\widehat{\phi_{t,\eta}}(\zeta)= 0.  
\]
The wave packet transform of a Schwartz function $f$ is then a function on ${\Pp^\circ}$ defined by
\[
F^\circ(f)(u,t,\eta)= |\l f,\phi_{u,t,\eta}\r|.
\]
With the same notation as in \eqref{maxfctbd} for ${\mathcal I}_{f,p,Q}$, and introducing the corresponding  good set of parameters
\begin{equation} 
\label{goodcirc}
\mathsf{G}_{f,p,Q}^\circ = {\Pp^\circ}\setminus \bigcup_{I\in {\mathcal I}_{f,p,Q}}  { \T^\circ}(3I)  
 \end{equation}
 we have the following continuous parameter version of Proposition \ref{CET}.
 \begin{proposition}{\cite[Theorem 1]{DPOu2}} \label{CETcirc}Let  $Q\subset \R$ be a dyadic interval and $f$ be a Schwartz function. For any $1<p<2$, $q>p'$  there exists $N=N(p,q)$ and $K=K(p,q)$ such that
 \begin{equation} \label{cetcirc}    \left\|F^\circ (f\cic{1}_{3Q}) \cic{1}_{\mathsf{G}^\circ_{f,p,Q}}\right\|_{L^{q}({\Pp^\circ}, \sigma,\mathsf s^\circ)} \leq KA_N  |Q|^{\frac1q}   \l f\r_{3Q,p}.
\end{equation} 
\begin{remark} The above proposition is obtained by choosing $\lambda=|Q|^{-\frac1p}$ in \cite[Theorem 1]{DPOu2}. There are, however, two minor   discrepancies between the result of \cite{DPOu2} and the one recalled above. The first one is that, in definition \eqref{goodcirc}, the intervals ${\mathcal I}_{\mathrm{M}_1f,p,Q}$  are used  in place of ${\mathcal I}_{f,p,Q}$. This change is necessary in order to perform a reduction argument to compact support in $\eta$ of $F^\circ (f)$ see \cite[Section 7.3.1]{DPOu2}, and can thus  be avoided in the setup of Proposition \ref{CETcirc} since the parameter $\eta$ is already in a compact interval. The second difference is that the adapted family ${\Phi}$ used in \cite{DPOu2} to define the wave packet transform is obtained by applying dilation, translation and modulation symmetries to a fixed mother wave packet. However, the arguments of \cite{DPOu2} adapt naturally to the more general transform obtained from \eqref{adaptxi}. We leave the details for the interested reader.  \end{remark}
 \end{proposition} 
\subsubsection{Transference} 
For each $P \in \Pp$ define
\[
P^\circ:=\left\{ (u,t,\eta) \in  \Pp^\circ:
|I_P| \leq t \leq 2|I_P|,  \, u \in I_P, \, \eta \in \omega_{P}\right\}.
\]
Up to possibly splitting $\Pp$ into finitely many subcollections the sets $\{P^\circ:P \in \mathbb P\} $ are pairwise disjoint subsets of $\Pp^\circ$. Furthermore, the $\d u\d t \d \eta$-measure of $P^\circ$ is comparable to $|I_P|$ up to a constant factor. Let $f$ be a fixed Schwartz function and $\{\phi_{P_j}: P\in \mathbb P\}$ be chosen such that
\[
F_{j}(f) (P) \leq 2|\l f, \phi_{P_j}\r|=:\bar F_j(f)(P) \qquad \forall P \in \mathbb P.
\]
Then the family defined by $\phi_{u,t,\eta}=\phi_{P_j}$ for all $(u,t,\eta)\in P^{\circ} $,  $\phi_{u,t,\eta}=0$ if $(u,t,\eta)$ does not belong to any $ P^{\circ} $ is an adapted system.
 We claim that, if $F^\circ (f) $ is the corresponding wave packet transform
\begin{equation}
\label{transf1}
\mu\left(\mathsf{s}\left(\bar F_j(f) \cic{1}_{ \mathsf{G}_{f,p,Q}}\right)>C\lambda \right) \leq  \mu^\circ\left(\mathsf{s}^\circ\left( F^\circ(f) \cic{1}_{ \mathsf{G}^\circ_{f,p,Q}}\right)>\lambda \right) 
\end{equation}
which, by virtue of the above definitions and of Proposition \ref{CETcirc}, implies the estimate of Proposition \ref{CET}.  Let $\lambda$ be fixed and $L$ denote the right hand side of \eqref{transf1}. Let $ \{\T^\circ_j(I_j,\xi_j)\}$ be a countable collection of tents  such that
\[
\sum_{j}\sigma(I_j) \leq L + \eps, \qquad \sup_{\T^\circ} \mathsf{s}^\circ\left( F^\circ(f) \cic{1}_{ \mathsf{G}^\circ_{f,p,Q}} \cic{1}_{  E^{\circ c}}\right)(\T^\circ) \leq \lambda, \qquad E^\circ:= \bigcup_j \T^\circ_j
\]
Now, for each $j$, let $\T_j= \T_j(I_j,\xi_j)$ be the maximal tree of tritiles with top data $(I_j,\xi_j)$ same as $\T^\circ_j$ and set 
\[
E := \bigcup_j \T _j\implies \mu(E) \leq  \sum_{j}\sigma(I_j) \leq L+\eps.
\] 
To obtain \eqref{transf1} and conclude the proof it then  suffices to show that for all $\T\in \mathcal T$ we have
\begin{equation}
\label{transf2}
\mathsf{s}\left( \bar F_j(f) \cic{1}_{ \mathsf{G}_{f,p,Q}} \cic{1}_{  E^c}\right)(\T)\leq C\mathsf{s}^\circ\left( F^\circ(f) \cic{1}_{ \mathsf{G}^\circ_{f,p,Q}} \cic{1}_{  E^{\circ c}}\right)(\T^\circ)
\end{equation}
where $\T^\circ$ is the tent with same top data as $\T$. Let us verify this for the $L^2$ portion of the size $\mathsf s$. This is a consequence of the following observations
\begin{itemize}
\item if $P \in \T \setminus \T_1$ (i.e. $P$ belongs to the lacunary part), then $P^\circ \subset \T^\circ_\ell$
\item if $P\in E^{ c} \cap \mathsf{G}_{f,p,Q} $, then $\widetilde{P^\circ}:=P^\circ \cap E^{\circ c}\cap \mathsf{G}^\circ_{f,p,Q}$ has $\d u\d t \d \eta$-measure larger than $C^{-1}|I_P|$
\end{itemize}
of which we leave the verification to the reader,
and  of the computation
\[\begin{split} &\quad 
\sum_{\substack{P \in \T \setminus \T_1 \\ P \in  E^c \cap \mathsf{G}_{f,p,Q}  }} |I_P| |\bar F_j(f)(P)|^2\\ &= \sum_{\substack{P \in \T \setminus \T_1 \\ P \in  E^c \cap \mathsf{G}_{f,p,Q}  }} \frac{|I_P|}{\nu(\widetilde{P^\circ})} \int \displaylimits_{ \widetilde{P^\circ}} |F^\circ(f)(u,t,\eta)|^2 \, \d u \d t \d \eta \leq  C \int \displaylimits_{  \T^\circ_\ell \cap E^{\circ c}\cap \mathsf{G}^\circ_{f,p,Q}} |F^\circ(f)(u,t,\eta)|^2 \, \d u \d t \d \eta \end{split}
\]
where we have denoted by $\nu$ the $\d u \d t \d \eta $ measure. The proof is complete.

\section{Proof of Theorem \ref{ThmDisc}} \label{Secpf}
Now we are ready to prove Theorem \ref{ThmDisc},  to which   Theorem \ref{ThmMain} has been reduced.  Since for any open admissible tuple $\vec{r}$ there exists an open admissible tuple $\vec p$ with $\max\{p_j\}<2$ and $p_j\leq r_j$, it suffices to prove the case $\max\{p_j\}<2$. Such a tuple $\vec p$ is fixed from now on.
\subsection{Construction of the sparse collection} \label{Ssgrid}
 Let $f_j\in L^{p_j}(\R)$, $j=1,2,3$, be three compactly supported functions and $\mathcal D$ be a dyadic grid. 
For all $Q\in \mathcal D$, referring to the notation \eqref{maxfctbd} for ${\mathcal I}_{f,p,Q}$  we may then define
 \begin{equation}
\label{Jfpq}
\mathcal{I}_{\vec f,\vec p,Q} := \textrm{  maximal elements of }\bigcup_{j=1}^{3}  {\mathcal I}_{f_j,p_j,Q}
\end{equation}  
It is clear that the intervals $I\in{\mathcal I}_{\vec f,\vec p,Q}$ are pairwise disjoint  and that 
\begin{equation}\label{mainlemQ}
\sum_{I\in {\mathcal I}_{\vec f,\vec p,Q}}|I|\leq\frac{|Q|}{2}.
\end{equation}
Furthermore, as a consequence of  \eqref{doublingcond} for each $f=f_j,p=p_j$, there holds
\begin{equation}
\label{doublingcondj}
\inf_{x\in 3I}\mathrm{M}_{p_j}f_j(x)\lesssim \l f\r_{3Q,p} \qquad \forall I \in {\mathcal I}_{\vec f,\vec p,Q}, \, j=1,2,3.
\end{equation}
  We now  put together these stopping intervals in  a  single sparse collection  $\mathcal S=\mathcal S(\mathcal D,f_1,f_2,f_3)$  of stopping intervals for the condition \eqref{maxfctbd}. Let us begin by choosing a partition of $\R$ by intervals $\{Q_k\in \mathcal D:k \in \mathbb N\}$ with the property that  $\supp f_j \subset 3Q_k$ for all $j=1,2,3$ and $k \in \mathbb N.$
For each $k$, let 
\[
\mathcal S(Q_k) =\bigcup_{\ell=0}^\infty \mathcal S_\ell(Q_k)
\]
where $\mathcal S_0(Q_k)=\{Q_k\}$ and, proceeding iteratively, \[
\mathcal S_\ell(Q_k)=\bigcup_{Q \in \mathcal{S}_{\ell-1}(Q_k)  }{\mathcal I}_{\vec f,\vec p,Q}, \qquad l=1,2,\ldots\]
Finally, define 
\[
\mathcal S=\mathcal S(\mathcal D,f_1,f_2,f_3)=\bigcup_{k=0}^\infty \mathcal S(Q_k).
\]
By construction and by the packing property \eqref{mainlemQ}, $\mathcal S$ is a $\frac12$-sparse  subcollection of $\mathcal D$.
%$$
%\mathcal{D}_j=\left\{2^{k}[0,1)+
%\left({\textstyle n+ \frac j3}\right)2^{k}: k,n\in \mathbb Z \right\}, \qquad j=0,1,2
%$$
%be the three canonical shifted dyadic grids on $\R$.

\subsection{Reduction to a single shifted dyadic grid} \label{Ssred}It is convenient to reduce to a canonical choice of dyadic grids, as follows. Let 
$$
\mathcal{D}_j=\left\{2^{k}[0,1)+
\left({\textstyle n+ \frac j3}\right)2^{k}: k,n\in \mathbb Z \right\}, \qquad j=0,1,2
$$
be the three canonical shifted dyadic grids on $\R$.  Recall  the well known fact that for all intervals $I\subset \R$ there exists a unique $\tilde I \in \mathcal{D}_0 \cup \mathcal{D}_1\cup \mathcal{D}_2$ with 
$3I\subset \tilde I$, $|\tilde I|\leq 6\cdot |3I|$,  and $c(\tilde I)$ is least possible. We say that   $I$ has type $j\in \{0,1,2\}$ if $\tilde I\in \mathcal D_j$. 

Fix a finite rank 1 collection $\Pp$ and a tuple of functions $\vec f=(f_1,f_2,f_3)$ as above. We split $\Pp=\Pp_0\cup \Pp_1\cup \Pp_2$ where $\Pp_j=\{P\in \Pp: I_P$ has type $j\}$. 
For each $j\in \{0,1,2\}$ we use the previous construction with $\mathcal D=\mathcal D_j$ to obtain a $\frac 1 2$-sparse collection of intervals $\mathcal S_j=\mathcal S(\mathcal D_j,\vec f)$ such that
\begin{equation}
\label{mainpf1}
 \Lambda_{\Pp_j}(f_1,f_2,f_3)\leq K A_N  \sum_{Q\in \mathcal{S}_j}|3Q|\prod_{\ell =1}^3\l f_{\ell}\r_{3Q,p_{\ell}}.
\end{equation}
Once \eqref{mainpf1} is performed, we achieve the estimate
\[
\begin{split}
& \Lambda_{\Pp}(f_1,f_2,f_3)= \sum_{j=0}^2  \Lambda_{\Pp_j}(f_1,f_2,f_3)\\
\leq & \sum_{j=0}^2 K A_N \sum_{Q\in \mathcal{S}_j}|3Q|\prod_{\ell =1}^3\l f_{\ell}\r_{3Q,p_{\ell}} \lesssim K A_N  \psf_{\widetilde{\mathcal S}}^{\vec p}(f_1,f_2,f_3),
 \end{split}
\]
where $\widetilde{\mathcal S}=\{3Q:\,Q\in \mathcal S_{j_0}\}$ and $j_0\in\{0,1,2\}$ is such that the right hand side of \eqref{mainpf1} is maximal.
Since $\mathcal S_{j_0}$ is $\frac12$ sparse it immediately follows that $\widetilde{\mathcal S}$ is a $\frac16$-sparse collection. This completes the proof of Theorem \ref{ThmDisc}, up to \eqref{mainpf1}. In the next three subsections, we give the proof of \eqref{mainpf1}.

 \subsection{Proof of \eqref{mainpf1}: main argument} 
 A first observation is that, since the intervals $I_P$ and $\widetilde{I_P}$ are comparable, and in view of the maximal definition of the tritile maps, there is no loss in generality in what follows to assume $I_P =\widetilde{I_P} $ for all $P \in \Pp_j$, that is $\{I_P: P \in \Pp\}\subset \D_j$. In fact, we are free to work with $j=0$ and accordingly forgo the subscript $j$ till the end of this section.

The main step of the argument for Theorem \ref{ThmMain} is summarized in the next lemma, whose proof is postponed to the next subsection. 
Before the statement, it is convenient to recall  the notation
\[
\Pp_{\leq} (Q):=\{P \in \mathbb P: I_P \subset Q\}
\]
associated to a generic finite collection of tritiles $\mathbb P$. 
Let $\{Q_k: k \in \mathbb N\}$ be the intervals employed in the construction of $\mathcal S$ in Subsection \ref{Ssgrid}. Since $\{Q_k: k \in \mathbb N\}$ partition $\R$,  we have the splitting 
\[
\Pp=\bigcup_{k=0}^\infty \Pp_{\leq} (Q_k);
\]
in fact the union is finite, as the collection $\Pp$ is. Since $\mathcal S=\cup_{k} \mathcal S(Q_k)$,  \eqref{mainpf1} is a consequence of 
\begin{equation}
\label{mainpf2}
\Lambda_{\Pp_{\leq}(Q_k)}(f_1,f_2,f_3)\leq   K A_N \sum_{Q\in \mathcal{S}(Q_k)}|3Q|\prod_{j =1}^3\l f_{j}\r_{3Q,p_{j}}. \end{equation}
Estimate \eqref{mainpf2} is obtained  by iteration of the lemma below, starting with $Q= Q_k$, which is legitimate because $\supp f_j\subset 3Q_k$ for any $j=1,2,3$, and following the construction of $\mathcal S(Q_k)$.
\begin{lemma}\label{LemmaMain} Let $\vec f=(f_1,f_2,f_3)$ be as above and   $Q\in \mathcal D$.
For any rank 1 collection  of tritiles $\mathbb P$  such that $\{I_P: P \in   \Pp\}\subset \mathcal D$, there holds
\[
 \Lambda_{\mathbb P}(f_1\cic{1}_{3Q},f_2\cic{1}_{3Q},f_3\cic{1}_{3Q})  \lesssim   |3Q|\prod_{j=1}^3 \l f_j\r_{3Q,p_j}+\sum_{I\in{\mathcal I}_{\vec f,\vec p,Q}}  \Lambda_{\mathbb P_{\leq}(I)} (f_1\cic{1}_{3I},f_2\cic{1}_{3I},f_3\cic{1}_{3I}) .
\]
\end{lemma}
We are left with the task of  showing that Lemma  \ref{LemmaMain} holds true.

%We turn to the actual proof of Theorem \ref{ThmMain}. Let there be given a tuple $\vec f=(f_1,f_2,f_3)\in \mathcal C^\infty(\R)^3$. Choose an interval $Q\in \mathcal D_0$ (standard dyadic grid on $\R$) large enough so that $\mathrm{supp}\, f_j\subset 3Q$ for $j=1,2,3$.  We will construct a collection of i

\subsection{Proof of Lemma \ref{LemmaMain}}
For the sake of brevity, we assume that all $f_j$'s are supported on $3Q$. With reference to  (\ref{Jfpq}) for  ${\mathcal I}_{\vec f,\vec p,Q}$, let 
\[
\mathsf{G}_{\vec{f},\vec{p},Q}:=\Pp\setminus \left(\bigcup_{I \in {\mathcal I}_{\vec f,\vec p,Q}} { \Pp_\leq}(I)\right).
\]We decompose
\[
\Lambda_\Pp(f_1,f_2,f_3)\leq\sum_{P\in \mathsf{G}_{\vec{f},\vec{p},Q}}|I_P|\prod_{j=1}^3F_j(f_j)(P)+\sum_{I \in {\mathcal I}_{\vec f,\vec p,Q}} \Lambda_{\Pp_\leq(I)}(f_1,f_2,f_3).
\]
We claim that the first term satisfies the following estimate:
\begin{equation}\label{mainlemLC}
\sum_{P\in \mathsf{G}_{\vec{f},\vec{p},Q}}|I_P|\prod_{j=1}^3F_j(f_j)(P)\lesssim |Q|\prod_{j=1}^3 \l f_j\r_{3Q,p_j}.
\end{equation}
Indeed, since $\vec{p}$ is open admissible and $\max\{p_j\}<2$, using  Proposition \ref{CET}  we learn that   there exists a H\"older tuple $\vec{q}$ such that $F_j$ has the (LC$_{q_j,p_j}$) property, $j=1,2,3$, i.e.
\[
\left\|F_j (f_j) \cic{1}_{\mathsf{G}_{f_j,p_j,Q}}\right\|_{L^{q_j}(\Pp, \sigma,\mathsf{s}_j)} \lesssim |Q|^{\frac{1}{q_j}}\l f_j\r_{3Q,p_j}.
\]Let $G_j(P):=F_j(f_j)(P)\cic{1}_{\mathsf{G}_{\vec{f},\vec{p},Q}}(P)$. By the H\"older inequality of Lemma \ref{Holder}, we have that
\begin{equation}\label{mainlem1}
\sum_{P\in \mathsf{G}_{\vec{f},\vec{p},Q}}|I_P|\prod_{j=1}^3F_j(f_j)(P)\lesssim \prod_{j=1}^3\|G_j\|_{L^{q_j}(\Pp,\sigma,\mathsf{s}_j)}.
\end{equation}
Now, since $\mathsf{G}_{\vec{f},\vec{p},Q}\subset \mathsf{G}_{f_j,p_j,Q}$ for $j=1,2,3$, we have
\[
\|G_j\|_{L^{q_j}(\Pp,\sigma,\mathsf{s}_j)}\leq\|F_j(f_j)\cic{1}_{\mathsf{G}_{f_j,p_j,Q}}\|_{L^{q_j}(\Pp,\sigma,\mathsf{s}_j)}\lesssim |Q|^{\frac{1}{q_j}}\l f_j\r_{3Q,p_j}.
\]
Inserting the above three inequalities into (\ref{mainlem1}) yields (\ref{mainlemLC}).

We are left with estimating the second term $$\sum_{I \in {\mathcal I}_{\vec f,\vec p,Q}} \Lambda_{\Pp_\leq(I)}(f_1,f_2,f_3),$$ for which we claim
\begin{equation}\label{mainlemtail}
\sum_{I\in{\mathcal I}_{\vec f,\vec p,Q}} {\Lambda_{\Pp_\leq (I)}}(f_1,f_2,f_3) \lesssim |Q|\prod_{j=1}^3 \l f_j\r_{3Q,p_j}+\sum_{I\in{{\mathcal I}_{\vec f,\vec p,Q}}}  {\Lambda_{\Pp_\leq (I)}} (f_1\cic{1}_{3I},f_2\cic{1}_{3I},f_3\cic{1}_{3I}) .
\end{equation} 
To see this, for each $I\in{{\mathcal I}_{\vec f,\vec p,Q}}$, define
\[
{\Lambda_{\Pp_\leq (I)}}^{\vec{t}}(f_1,f_2,f_3):={\Lambda_{\Pp_\leq (I)}}(f_1\cic{1}_{I^{t_1}},f_2\cic{1}_{I^{t_2}},f_3\cic{1}_{I^{t_3}})=\sum_{P\in\Pp_{\leq}(I)}|I_P|\prod_{j=1}^3 F_j(f_j\cic{1}_{I^{t_j}})(P),
\]where $\vec{t}=(t_1,t_2,t_3)\in\{\mathsf{in},\mathsf{out}\}^3$ and
\[
I^{\mathsf{in}}:=3I,\qquad I^{\mathsf{out}}:=\R\setminus 3I.
\]Therefore, one can split
\[
{\Lambda_{\Pp_\leq (I)}}(f_1,f_2,f_3)\leq \sum_{\vec{t}\in\{\mathsf{in},\mathsf{out}\}^3}{\Lambda_{\Pp_\leq (I)}}^{\vec{t}}(f_1,f_2,f_3).
\]
Among the $2^3$ forms on the right hand side, the one corresponding with $\vec{t}$ such that $t_j=\mathsf{in}$ for all $j$ appears exactly in the second term on the right hand side of (\ref{mainlemtail}), hence it suffices for us to bound the rest of the $2^3-1$ forms. According to Proposition \ref{tailprop}, which we state and prove later, for any $\vec{t}$ such that $t_j=\mathsf{out}$ for at least one $j=1,2,3$, there holds
\[
 {\Lambda_{\Pp_\leq (I)}}^{\vec{t}}(f_1,f_2,f_3) \lesssim |I|\prod_{j=1}^3 \inf_{x\in 3I} \mathrm{M}_{p_j}f_j(x)\lesssim |I|\prod_{j=1}^3 \l f_j\r_{3Q,p_j}
\]
where the last step follows from \eqref{doublingcondj}. Therefore, multiplying the three inequalities together and summing over $I$ yields (\ref{mainlemtail}), which also completes the proof of the lemma.

\subsection{Handling the tail terms} 
Now we proceed with the proposition that has been used in the proof of Lemma \ref{mainlemQ} to estimate the tail term. In fact, we are going to derive it in a more general form, which not only includes our tritile maps $F_j$ as a special case, but also applies to more general tritile maps. A tritile map $F: L^{1}_{\mathrm{loc}} (\R) \to \mathbb C^{\Pp}$ is said to be \emph{almost localized} if it satisfies
\begin{equation}
\label{allocmultmap}
\begin{split} 
\sup_{P \in \Pp_=(J)}F(f)(P)& \lesssim  \|f\|_{L^1(\chi_J^M)}, \\
\left(\frac{1}{|J|}\sum_{P \in \Pp_=(J)}|I_P|F(f)(P)^2\right)^{\frac12}& \lesssim \|f\|_{L^2(\chi_J^M)},
\end{split}
\end{equation}
where $M$ is a fixed large integer (say $M=10^3$), and we have used the notation $\Pp_=(J):=\{P\in\Pp: I_P=J\}$.

\begin{proposition}\label{tailprop}
Assume the type $\vec{t}$ is such that $t_j=\mathsf{out}$ for at least one $j=1,2,3$. Let $F_j$ be almost localized tritile maps for $j=1,2,3$, and $\vec{p}$ be an open admissible tuple. Then,
\begin{equation}
\label{2tailprop}
\Lambda_{\Pp_\leq (I)}^{\vec t}(f_1,f_2,f_3) \lesssim |I| \prod_{j=1}^3 \inf_{x \in 3I} \mathrm{M}_{p_j} f_j (x).
\end{equation}
\end{proposition}
The proof of the proposition will rely on the following key lemma.
\begin{lemma}
\label{almloc2taillemma}
Let $J$ be an interval. Assume that $\supp f_3 \cap AJ=\emptyset$ for some $A\geq 3$. Let $\vec{p}$ be an open admissible tuple and $F_j$ be an almost localized tritile map for j=1,2,3. Then 
\[
 \Lambda_{\Pp_{=}(J)} (f_1,f_2,f_3):= \sum_{P \in \Pp_{=}(J)} |I_P| \prod_{j=1}^3 F_j(f_j )(P)    \lesssim A^{-100}|J| \prod_{j=1}^3 \inf_{x \in 3J} \mathrm{M}_{p_j} f_j (x).
\]
\end{lemma}
\begin{proof}
The almost localized assumptions \eqref{allocmultmap} can be rephrased in the form
 \[
\|F_j(f)(\cdot)\|_{\ell^\infty(\Pp_=(J))} \lesssim \|f\|_{L^1(\chi_J^M)}, \qquad \|F_j(f)(\cdot)\|_{\ell^2(\Pp_=(J))} \lesssim \|f\|_{L^2(\chi_J^M)}, \]
which by off-diagonal Marcinkiewicz interpolation yields for $1\leq p\leq 2$
\begin{equation}
\label{taillemmapf1}
\|F_j(f)(\cdot)\|_{\ell^{p'}(\Pp_=(J))} \lesssim \|f\|_{L^p(\chi_J^M)}.\end{equation}
Notice that for $j=1,2$, 
\begin{equation}
\label{taillemmapf2}
\|f_j\|_{L^p(\chi_J^M)}\lesssim 
\inf_{x \in 3J} \mathrm{M}_{p} f_j (x)
\end{equation}
while if $M$ is sufficiently large 
\begin{equation}
\label{taillemmapf3}
\|f_3\|_{L^p(\chi_J^M)}\lesssim\left( \sup_{x \in \supp f_3} \chi_J^{100}(x)\right)\|f_3\|_{L^p(\chi_J^{M-100})}\lesssim
A^{-100} \inf_{x\in 3J}\mathrm{M}_{p} f_3 (x).
\end{equation}
Since $\vec p$ is open admissible, there exists a H\"older tuple $\vec q=(q_1,q_2,q_3)$ with $(q_j)'\leq p_j $. Therefore, using \eqref{taillemmapf1} for each $f=f_j$
\[
\Lambda_{\Pp_{=}(J)} (f_1 ,f_2,f_3)  \leq |J| \prod_{j=1}^3 \|F_j(f)(\cdot)\|_{\ell^{q_j}(\Pp_=(J))} \lesssim A^{-100}|J| \prod_{j=1}^3 \inf_{x \in 3J} \mathrm{M}_{{q_j'}} f_j (x),
\]
which is stronger than the estimate claimed of the lemma.
\end{proof}

\begin{proof}[Proof of Proposition \ref{tailprop}] For the sake of definiteness, let us assume that $t_3=\mathsf{out}$ and  Let ${\mathcal J}=\{J:J=I_P$ for some $P \in \Pp_{\leq}(I) \}$. We partition 
\[
{\mathcal J}_k=\{J\in {\mathcal J}: 2^{k}J \subset I, 2^{k+1}J\not \subset I \}, \qquad \Pp_{\leq,k}(I)=\{P \in \Pp_{\leq}(I): I_P\in {\mathcal J}_k\}
\]
Let us observe the following properties of the intervals $J \in {\mathcal J}_k$:
\begin{equation}
\label{Jks}
\begin{split}
&\dist(J, \supp f_3 \cic{1}_{I^\mathsf{out}}) \sim 2^{k}|J|, \\
&J\in{\mathcal J}_k \textrm{ have finite overlap  and } \sum_ {J \in {\mathcal J}_k}|J|\lesssim |I|,\\
& \inf_{x \in 3J} \mathrm{M}_{{p_j}} f_j (x) \lesssim 2^k \inf_{x \in 3I} \mathrm{M}_{{p_j}} f_j (x).
\end{split}
\end{equation}
We then estimate, using Lemma \ref{almloc2taillemma} and the above properties
\[
\begin{split}
&\quad \Lambda_{\Pp_{\leq}(I)}^{\vec t}(f_1,f_2,f_3) \leq \sum_{k\geq 0 } \sum_{J \in {\mathcal J}_k} \Lambda_{\Pp_{=}(J)} (f_1\cic{1}_{I^{t_1}} ,f_2 \cic{1}_{I^{t_2}}, f_3\cic{1}_{I^\mathsf{out}})\\ & \lesssim  \sum_{k\geq 0 } \sum_{J \in {\mathcal J}_k}  2^{-100k}|J| \prod_{j=1}^3 \inf_{x \in 3J} \mathrm{M}_{p_j} f_j (x) \lesssim  |I| \prod_{j=1}^3 \inf_{x \in 3I} \mathrm{M}_{p_j} f_j (x).
\end{split}
\]
The proof of the proposition is thus completed.
\end{proof}

Now that we have proved Proposition \ref{tailprop}, in order to complete the proof of Lemma \ref{LemmaMain}, it suffices to verify that the tritile maps $F_j$, $j=1,2,3$ given in (\ref{tritilemaps}) are indeed almost localized.
\begin{lemma}\label{checkalmloc}
Tritile maps
\[
F_j(f)(P)=\sup_{\phi_{P_j} \in \cic{\Phi}(P_j) } |\l f_j, \phi_{P_j}  \r|, \qquad j=1,2,3
\]are almost localized. {I}n other words, 
\begin{equation}
\label{checkalmloc1} 
\sup_{P \in \Pp_=(J)}F_j(f)(P)\lesssim  \|f\|_{L^1(\chi_J^M)},
\end{equation}and
\begin{equation}
\label{checkalmloc2}
\left(\frac{1}{|{J}|}\sum_{P \in \Pp_=({J})}|{I_P}|F_j(f)(P)^2\right)^{\frac12} \lesssim \|f\|_{L^2(\chi_{J}^M)}.
\end{equation}
\end{lemma}

\begin{proof}
To see (\ref{checkalmloc1}), for any $P\in \Pp_=({J})$ and $\phi_{P_j}\in \cic{\Phi}(P_j)$, write
\[
|\l f, \phi_{P_j}\r|=|\l f \chi_{J}^M|{J}|^{-1},\phi_{P_j}\chi_{J}^{-M} |{J}|\r|\leq \left\|f \chi_{J}^M\right\|_{L^1(\chi_{J}^M)}\left\|\phi_{P_j}\chi_{J}^{-M} |{J}|\right\|_{L^\infty}.
\]
Then according to (\ref{adapt}), (\ref{checkalmloc1}) follows immediately from $\left\|\phi_{P_j}\chi_{J}^{-M} |{J}|\right\|_{L^\infty}\leq A_M$.

Now we verify that (\ref{checkalmloc2}) holds true. Without loss of generality, one can assume that there exists $\{\phi_{P_j}\}$ such that the supremum in the definition of $F_j$ are attained up to an $\epsilon$. This can certainly be done if the collection $\Pp$ is finite. Since our estimate will not depend on the cardinality of the collection, a limiting argument will pass this to the infinite collection case as well. Hence, we are now trying to show that
\[
\left(\frac{1}{|{J}|}\sum_{P\in\Pp_=({J})}|{I_P}||\l f,\phi_{P_j}\r|^2\right)^{\frac{1}{2}}\lesssim \|f\|_{L^2(\chi_{J}^M)}.
\]To see this, write
\[
|{I_P}||\l f,\phi_{P_j}\r|^2=|\l f\chi_{J}^M,\phi_{P_j}\chi_{J}^{-M}|{I_P}|^{1/2}\r|^2.
\]Define $\tilde{\phi}_{P_j}:=|{I_P}|^{1/2}\phi_{P_j}\chi_{J}^{-M}$. We claim that $\{\tilde{\phi}_{P_j}\}$ is an orthogonal system with $L^2$ normalization, which yields (\ref{checkalmloc2}) immediately.

The $L^2$ normalization can be easily seen from
\[
\int_{\R} |\tilde{\phi}_{P_j}|^2(x)\,dx\leq A_{M+1}|{J}|^{-1}\int_{\R}\chi_{J}^2(x)\,dx\lesssim A_{M+1}.
\]And the orthogonality follows from the disjoint frequency supports consideration of $\{\phi_{P_j}\}$. More precisely, since ${I_P}={J}$ for all $P\in\Pp_=({J})$, $\{\supp \widehat{\phi}_{P_j}\subset \omega_{P_j}\}$ are pairwise disjoint. Therefore, since the Fourier transform of $\tilde{\phi}_{P_j}$ is  a finite linear combination of derivatives (up to order $2M$) of the Fourier transform of $\phi_{P_j}$,  $\tilde{\phi}_{P_j} $ and  ${\phi}_{P_j} $  have the same frequency support, which implies the desired orthogonality.
\end{proof}

\section{Proof of    Theorem \ref{MultApThm} and Corollary \ref{Aqcor}} \label{SecWeight}

\subsection{Proof of Theorem \ref{MultApThm}}   
Fixing a  tuple $\vec q=(q_1,q_2,q_3)$ and weights $\vec v=(v_1,v_2,v_3) $ as in the statement of the theorem,  and any open admissible tuple $\vec p$ with $p_j<q_j$ for $j=1,2,3$, proving the theorem amounts  to showing that
\begin{equation}
\label{W1}
 \sup_{m } \left|\Lambda_m(f_1,f_2,f_3)\right| \leq  K(\vec p, \vec q, \vec v)
 \prod_{j=1}^3 \|f_j\|_{L^{q_j}(v_j)} 
 \end{equation}
 where  $ K(\vec p, \vec q, \vec v)$ is the constant appearing in the statement of the theorem, holds for all tuples $\vec f=(f_1,f_2,f_3)\in \mathcal C^\infty(\R)^3$. We define
 \[
 w_j= {v_j}^{\frac{p_j}{p_j-q_j}}, \qquad j=1,2,3.
 \]
 Note that the finiteness of  the ${A_{\vec q}^{\,\vec p}}$ constant of $\vec v$ implies $w_j\in L^1_{\mathrm{loc}}(\R)$. Setting $f_j=g_jw_j^{\frac{1}{p_j}}$ one notices that $\|f_j\|_{L^{q_j}(v_j)} = \|g_j\|_{L^{q_j}(w_j)}$.   Applying the domination result from Theorem \ref{ThmMain}, we bound the left hand side of \eqref{W1} by
\[ \sup_{\mathcal S \textrm{ sparse}} \mathsf{PSF}_{\mathcal S}^{\vec p}(f_1,f_2,f_3)=\sup_{\mathcal S \textrm{ sparse}} \mathsf{PSF}_{\mathcal S}^{\vec p}\big(g_1w_1^{\frac{1}{p_1}},g_2w_2^{\frac{1}{p_2}},g_3w_3^{\frac{1}{p_3}}\big).
 \]
 {By possibly splitting $\mathcal S$ into three subcollections and using the three grid lemma recalled in Subsection \ref{Ssred}, we can restrict to the case of $\mathcal S$ being a sparse subset of the standard dyadic grid $\mathcal D_0$.}
 Therefore, \eqref{W1} will follow from the estimate of the lemma below.
\begin{lemma} \label{MAINWEIGHT}  For any $g_j\in L^{q_j}(w_j)$, $j=1,2,3$, there holds
\[
\sup_{\mathcal S\subset \mathcal D_0 \;\frac16-\mathrm{ sparse}} \mathsf{PSF}_{\mathcal S}^{\vec p}(g_1w_1^{\frac{1}{p_1}},g_2w_2^{\frac{1}{p_2}},g_3w_3^{\frac{1}{p_3}}\big)\lesssim   
\mu_{\vec p,\vec q} [\vec v]_{A_{\vec q}^{\,\vec p}}^{\max\left\{\frac{q_j}{q_j-p_j}\right\}}   
 \prod_{j=1}^3 \|g_j\|_{L^{q_j}(w_j)}
 \]
 where\[\mu_{\vec p,\vec q}:= 
 \left(  \prod_{j=1}^{3} \frac{q_j}{q_j-p_j} \right)2^{3\left(\sum_{j=1}^3 \frac{1}{p_j}-1\right) \max\left\{\frac{p_j}{q_j-p_j}\right\}}.\]
\end{lemma}
%We postpone the proof of this lemma and turn to the removal of the assumption  $w_j\in \mathcal C^\infty(\R)$. Fix three smooth compactly supported  functions $g_1,g_2,g_3$. Let $R$ be large enough such that $B=\{|x|<R\}$ contains the supports of $g_j$. Then there exist sequences $w_{jn}\in \mathcal C^\infty(\R)$ converging to $w_j$ in $L^1(B)$. Then $g_j^{q_j} w_{jn}$ converges to $g_j^{q_j} w_{jn}$ in $L^1(B)$ as well. By using this fact and the estimate \eqref{W1} for  $w_j=w_{jn}$, then
%\[
%\Lambda_m\big(g_1w_1^{\frac{1}{p_1}},g_2w_2^{\frac{1}{p_2}},g_3w_3^{\frac{1}{p_3}}\big):=\lim_{n\to \infty} \Lambda_m\big(g_1w_{1n}^{\frac{1}{p_1}},g_2w_{2n}^{\frac{1}{p_2}},g_3w_{3n}^{\frac{1}{p_3}}\big)
%\]
%is well defined and bounded   by the right hand side of \eqref{W1}. The proof of Theorem \ref{MultApThm} is complete, up to showing that Lemma \ref{MAINWEIGHT} holds true. 
\begin{proof} We largely follow the argument from \cite{LerNaz2015}. We may work with $g_j\geq 0$. Let $\mathcal S$ be a fixed $1/2$-sparse grid. Then 
\begin{equation}
\label{W2}
\begin{split}
&\quad \mathsf{PSF}_{\mathcal S}^{\vec p}(g_1w_1^{\frac{1}{p_1}},g_2w_2^{\frac{1}{p_2}},g_3w_3^{\frac{1}{p_3}}\big)= \sum_{Q\in \mathcal S} |Q|\prod_{j=1}^3\left(\l g_j ^{p_j} w_j\r_Q\right)^{\frac{1}{p_j}} \\  &= \sum_{Q\in \mathcal S} \left(\prod_{j=1}^3 w_j(E_Q)^{\frac{1}{q_j}}
\left(\frac{\l g_j ^{p_j} w_j\r_Q}{\l w_j\r_Q}\right)^{\frac{1}{p_j}} \right) \times \left(\prod_{j=1}^3  {\l w_j\r_Q}^{\frac{1}{p_j}-\frac{1}{q_j}} \right) \times \left(|Q|\prod_{j=1}^3 \left( \frac{{\l w_j\r_Q}}{w_j(E_Q)} \right)^{\frac{1}{q_j}} \right).
  \end{split}
\end{equation}
 The second  product inside the sum of  \eqref{W2} is the precursor to $[\vec v]_{A_{\vec q}^{\,\vec p}}$. Arguing as in \cite{LerNaz2015}, the rightmost factor in \eqref{W2} is bounded above uniformly in $Q$ by
\[
2^{3\left(\sum_{j=1}^3 \frac{1}{p_j}-1\right) \max\left\{\frac{p_j}{q_j-p_j}\right\}} [\vec v]_{A_{\vec q}^{\,\vec p}}^{\max \left\{\frac{1}{p_jq_j}\right\}}.
\]
Introducing the dyadic weighted maximal functions 
\[
\mathrm{M}_{p_j,w_j}(f)(x)  = \sup_{Q \in \mathcal D_0} \left(\frac{\l |f| ^{p_j} w_j\r_Q}{\l w_j\r_Q}\right)^{\frac{1}{p_j}} \cic{1}_{Q}(x)
\] and using the disjointness of $E_Q$ and H\"older's inequality, we estimate the remaining part of \eqref{W2} by
\[
\sum_{Q\in \mathcal S} \left(\prod_{j=1}^3 w_j(E_Q)^{\frac{1}{q_j}}
\left(\frac{\l g_j ^{p_j} w_j\r_Q}{\l w_j\r_Q}\right)^{\frac{1}{p_j}} \right) \leq \prod_{j=1}^{3} \|\mathrm{M}_{p_j,w_j}g_j\|_{L^{q_j}(w_j)}.
\]
The claimed estimate then follows by bookkeeping the last three observations and by relying upon the sharp $L^{q_j}(w_j)$-boundedness of  $\mathrm{M}_{p_j,w_j}(f_j)$ (see \cite{KM} for a proof). The proof of the lemma is complete.
\end{proof}
\subsection{Proof of Corollary \ref{Aqcor}}First of all, we use the openness of the $A_q$ and $RH_\alpha$  classes and the equivalence \cite{CruzMartellPerez}
\[
  w^s \in A_q \iff w \in A_{\frac{q+s-1}{s}} \cap RH_s
\]
 to find $\eps>0$ such that
\begin{equation}
\label{rhap}
[v_j^{\frac{2}{1-\eps}}]_{A_{q_j}} \leq \mathcal Q\left([v_j^2]_{A_{q_j}} \right), \qquad 
\end{equation}
where $\mathcal Q$ is a positive increasing function of its argument.
We denote by $q_3$ the dual exponent of $r$ and by $v_3=u_3^{1-q_3}$ the dual weight.  We will prove the corollary by  finding an open admissible tuple $\vec p$ with $p_j <q_j$ such that 
\begin{equation}
\label{W4}
[\vec v]_{A_{\vec q}^{\,\vec p}}\leq \prod_{j=1}^2  [v_j^{\frac{2}{1-\eps}}]_{A_{q_j}}^{\frac{1-\eps}{2q_j}} 
\end{equation}
and subsequently applying Theorem \ref{MultApThm}, which is made possible by  \eqref{rhap}.

Referring to the notation of \eqref{admtuple}, let $\vec{p} $ be an open admissible tuple with $p_j <q_j$ and 
$\eps=\eps(\vec p).
$
We set $\delta=1
+\eps$ and reparametrize
\begin{equation} \label{parameters}   \textstyle \frac{1}{p_j}= 1-\frac{\delta\theta_j}{r_j}, \qquad r_j= \frac{q_j}{q_j-1}, \qquad \theta_ j\geq 0, \qquad {\displaystyle\sum_{j=1}^3 }\frac{\theta_j}{r_j}=1.
\end{equation}
This leads to the following lemma.
\begin{lemma} \label{apbound} There holds
\[
[\vec v]_{A_{\vec q}^{\,\vec p}}\leq \prod_{j=1}^2 \sup_{Q\subset \R} \left(  \left\l v_j^{\frac{1}{ (1-\delta\theta_3)}} \right\r_Q^{{1-\delta \theta_3}}\left \l  v_j^{\frac{1}{ 1-\delta\theta_j} \frac{1}{1-q_j} } \right\r_Q^{ (q_j-1)(1-\delta \theta_j)}\right)^{\frac{1}{q_j}} 
\]
\end{lemma} 
\begin{proof}Observe that
$$
\textstyle 
\frac{1}{p_j}-\frac{1}{q_j}= \frac{1-\delta \theta_j}{r_j}
$$
Using the relation $1=v_1^{\frac{1}{q_1}}v_2^{\frac{1}{q_2}} v_3^{\frac{1}{q_3}}$, the definition of $w_j$ and H\"older, one has
\[\begin{split}
\l w_n \r_Q^{\frac{1}{p_3}-\frac{1}{q_3}}= \left \l \prod_{j=1}^2 v_j^{\frac{r_3}{q_j(1-\delta\theta_3)}} \right\r_Q^{\frac{1-\delta \theta_3}{r_3}}  \leq   \prod_{j=1}^2 \left\l v_j^{\frac{1}{ (1-\delta\theta_3)}} \right\r_Q^{\frac{1-\delta \theta_3}{q_j}}
\end{split}
\]
and for $j=1,2$
\[\begin{split}
\l w_j \r_Q^{\frac{1}{p_j}-\frac{1}{q_j}}= \left \l  v_j^{\frac{1}{ 1-\delta\theta_j} \frac{1}{1-q_j} } \right\r_Q^{\frac{1-\delta \theta_j}{r_j}}   
\end{split}
\]
which, rearranging and taking suprema, completes the proof of the lemma.
\end{proof} Now, comparing with \eqref{parameters}, we may choose $\theta_1=\theta_2=\theta_3=\frac12$ in Lemma \ref{apbound}. This leads to the estimate 
\[
[\vec v]_{A_{\vec q}^{\,\vec p}}\leq \prod_{j=1}^2 \sup_{Q\subset \R} \left(  \left\l v_j^{\frac{2}{ 2-\delta }} \right\r_Q \left\l  v_j^{\frac{2}{ 2-\delta } \frac{1}{1-q_j} } \right\r_Q^{ q_j-1 }\right)^{\frac{2-\delta}{2}\frac{1}{q_j}} 
\]
whose right hand side is   the same as that of \eqref{W4}. This completes the proof of Corollary \ref{Aqcor}.

\appendix
\section{Vector-valued estimates from sparse domination} \label{SecVV} In this section, the tuple $\vec r=(r_1,r_2,r_3)$ always satisfies
\begin{equation}
\label{rtuple}1< r_1,r_2,r_3\leq \infty, \qquad  \sum_{j=1}^3 {\textstyle \frac{1}{r_j}=1.}
\end{equation}
 We turn to the study of the   trilinear forms
\[
\Lambda_{\cic{ m}}(\cic{f}_1,  \cic{f}_2, \cic{f}_3):=\sum_{k}\Lambda_{m_k}(f_{1k},  f_{2k}, f_{3k})
\] acting on $\ell^{r_j}$-valued sequences  $\cic{f}_j=\{f_{jk}\}$, 
where $\cic{m}=\{m_k\}$ is a sequence of multipliers satisfying \eqref{decay} uniformly. The adjoints to the above trilinear forms are the sequence-valued bilinear operators
\begin{equation}
\label{TVV}T_{\cic{m}}(\cic{f}_1,  \cic{f}_2)= \{T_{m_k}(f_{1k},  f_{2k})\}.
\end{equation}
 A  consequence of Theorem \ref{ThmMain} and of the classical  
 Fefferman-Stein inequalities \cite{FS}
\begin{equation}
\label{FSI}\begin{split}
&\left\| \{\mathrm{M}_p f_k\} \right\|_{L^{q}(\R;\ell^r)} \leq C(p,q,r) \left\| \{ f_k\} \right\|_{L^{q}(\R;\ell^r)}, \qquad 1\leq p<\min\{q,r\},\quad \sup\{q,r\}<\infty\\  & 
\left\| \{\mathrm{M}_p f_k\} \right\|_{L^{p,\infty}(\R;\ell^r)} \leq C(p,r) \left\| \{ f_k\} \right\|_{L^{p}(\R;\ell^r)}, \qquad 1\leq p<r<\infty. 
\end{split}
\end{equation}
 are the following vector-valued estimates for the operators $T_{\cic m}$ of \eqref{TVV} 
\begin{corollary} \label{CorVV}  Let $\vec r$ be  a fixed tuple as in \eqref{rtuple} and $\cic m=\{m_k\}$ be a sequence of multipliers satisfying \eqref{decay} uniformly. Then
the  bilinear operator  $T_{\cic m}$ of \eqref{TVV}   has the mapping properties
\begin{equation}
\label{IN2VV}
T_{\cic{m}}: L^{q_1}(\R;\ell^{r_1}) \times L^{q_2}(\R;\ell^{r_2}) \to L^{\frac{q_1q_2}{q_1+q_2}} (\R;\ell^{s_3}), \qquad s_3:=\frac{r_3}{r_3-1} 
\end{equation}
for all exponent pairs $(q_1,q_2)$ satisfying
\begin{equation}
\label{Rangeqr}   
   1<\inf\{q_1,q_2\}<\infty, \qquad \sum_{j=1}^3\frac{1}{\min\{q_j,r_j,2\} }< 2, \qquad \frac{1}{q_3}:=\max\big\{\textstyle 1-\big(\frac{1}{q_1}+\frac{1}{q_2}\big),0\big\}. 
\end{equation}
\end{corollary}
For each pair $(q_1,q_2)$,  the range of tuples   $\vec r$ for which $T_{m}$ admits $L^{q_1}\times L^{q_2}$ a bounded vector-valued extension is the same as the one recently obtained in  \cite[Theorem 7]{BenMusc15}   for the vector valued bilinear Hilbert transforms. Condition \eqref{Rangeqr} needs to be imposed in order to ensure that the set   
\begin{equation}
\label{nonempty1VV} \big\{ \vec p=(p_1,p_2,p_3) \textrm{ open admissible}: p_j<\min\{r_j,q_j\},  j=1,2,3.\big\}
\end{equation} is nonempty.

 \subsection{Proof of Corollary \ref{CorVV}} By an approximation argument, there is no loss in generality  in working with multipliers $\cic{m}=\{m_k\}$ with $m_k=0$ for all but finitely many $k$.
 
 Fix a tuple $\vec r$ as in \eqref{rtuple}. We assume $\sup r_j <\infty$: the case $r_j=\infty$ for (at most one) $j$ requires only minor modifications.  We first prove the case where $(q_1,q_2)$ is an exponent pair satisfying \eqref{Rangeqr} with $q_3<\infty$. In this range $\vec q=(q_1,q_2,q_3)$ is a H\"older tuple and  the claimed estimate on $T_{\cic m}$ is equivalent to proving that 
\begin{equation}
\label{vvpf0}
\big|\Lambda_{\cic{m}}(\cic{f}_1,\cic{f}_2,\cic{f}_3)\big| \leq \sum_{k}|\Lambda_{m_k}(f_{1k},  f_{2k}, f_{3k})|
  \lesssim \prod_{j=1}^3 \|\cic{f}_j\|_{L^{q_j}(\R; \ell^{r_j})}
\end{equation} 
 Since the set \eqref{nonempty1VV} is nonempty, we may choose an open admissible tuple $\vec p= (p_1,p_2,p_3)$ with $p_j<\min\{q_j,r_j\}$. 
  We apply the domination Theorem \ref{ThmMain} to each $m_k$ in the above sum, yielding the existence of sparse collections $\mathcal S_k$ for which the estimate
\[
|\Lambda_{m_k}(f_{1k},  f_{2k}, f_{3k}) |\lesssim \psf^{\vec p}_{\mathcal S_k}(f_{1k},  f_{2k}, f_{3k}) \]
holds true. If $\{E_I: I\in \mathcal S_k\}$ are the distinguished pairwise disjoint major subsets of $I\in \mathcal S_k$, we have
\begin{equation}
\label{toholders}
\psf^{\vec p}_{\mathcal S_k}(f_{1k},  f_{2k}, f_{3k}) \lesssim  \sum_{I \in \mathcal S_k} |E_I| \left( \prod_{j=1}^3\inf_{x \in E_I} \mathrm{M}_{p_j} f_{jk}(x) \right) 
\lesssim \int_{\R}\left( \prod_{j=1}^3 \mathrm{M}_{p_j} f_{jk}(x) \right) \,\d x
\end{equation}
Summing over $k$ and using H\"older's inequality first for the tuple $\vec r$ in the sum, and later for the H\"older tuple $\vec q$ in the integral the left-hand side of \eqref{vvpf0} is bounded by
\[
\begin{split}
  & \int_{\R}\sum_k\left( \prod_{j=1}^3 \mathrm{M}_{p_j} f_{jk}  \right)  \d x \leq  \int_\R \left( \prod_{j=1}^3 \big\|\{\mathrm{M}_{p_j} f_{jk} \}\big\|_{\ell^{r_j}}\right) \d x 
  \leq \prod_{j=1}^3 \big\|\{\mathrm{M}_{p_j} f_{jk}\}\big\|_{L^{q_j}( \ell^{r_j})} \lesssim \prod_{j=1}^3 \big\|\cic{f}_j \big\|_{L^{q_j}(\ell^{r_j})},
\end{split}
\]
having employed the Fefferman-Stein inequality \eqref{FSI} in the last step. This completes the proof of the case $q_3<\infty$.

We pass to the case $q_3=\infty$. In this range, we are able to choose an open admissible tuple with
\[
p_1<\min\{q_1,r_1\}, \quad p_2<\min\{q_2,r_2\}, \quad p_3<\min\{2,r_3\}.
\]
Also, by virtue of the fact that $1/q_1+1/q_2>1$, we can find a tuple of exponents $\vec t=(t_1,t_2,t_3)$ satisfying 
\begin{equation}
\label{rangeVV3}
t_1>q_1, \quad t_2>q_2, \quad t_3>p_3, \qquad \textstyle \frac{1}{t_1}+\frac{1}{t_2}+\frac{1}{t_3}=1.
\end{equation}
 Since the claimed range of exponents $(q_1,q_2)$ is open, it suffices to prove the weak-type analogue of \eqref{IN2VV} and then invoke multilinear vector-valued Marcinkiewicz interpolation. 
Such a weak-type estimate is equivalent to proving that for all $\cic{f}_j\in L^{q_j}(\R; 
 \ell^{r_j})$, $j=1,2$ of unit norm and for all sets $F_3\subset \R$ of finite measure, there exists $F_3'\subset F_3$ with $|F_3|\leq 2|F'_3|$ so that 
\begin{equation}
\label{vvpf1}
\big|\Lambda_{\cic{m}}(\cic{f}_1,\cic{f}_2,\cic{f}_3)\big| \leq \sum_{k}|\Lambda_{m_k}(f_{1k},  f_{2k}, f_{3k})|  \lesssim |F_3|^{1-\left(\frac{1}{q_1}+\frac{1}{q_2}\right)} \quad \forall{ \cic{f}_3} : \|\cic{f}_3(x)\|_{\ell^{r_3}} \leq \cic{1}_{F_3'}(x).
\end{equation}
Fix such $\cic{f}_1,\cic{f}_2,F_3$. We proceed with the definition of $F_3'$ in two steps. First, set
\[
H:=\bigcup_{j=1}^{2} \left\{ x \in \R: \|\{\mathrm{M}_{p_j} f_{jk}(x)\}\|_{\ell^{r_j}} > C |F_3|^{-\frac{1}{q_j}}\right\}
\]
By Chebychev  and  Fefferman-Stein inequalities \eqref{FSI}, 
$
|H| \leq 2^{-12}|F_3|
$
provided $C$ is chosen large enough. Then
\[
\widetilde H:=\bigcup_{Q \in \mathcal Q} 9Q, \qquad \mathcal Q=\left\{ \textrm{max.\ dyad.\ int.\ } Q : |Q \cap H|\geq 2^{-5} |Q|\right\}
\]
satisfies $ |\widetilde H|\leq 9\cdot 2^{5} |H| \leq 2^{-3} |F_3|$. Therefore the set $F_3':=F_3\backslash \widetilde H$ is a major subset of $F_3$. Fixing now any $\cic{f}_3=\{f_{3k}\}$ restricted to $F_3'$ as in \eqref{vvpf1},  we apply the domination  Theorem \ref{ThmMain} to each $m_k$ in \eqref{vvpf1}, yielding the existence of sparse collections $\mathcal S_k$ for which we have the estimate
\[
|\Lambda_{m_k}(f_{1k},  f_{2k}, f_{3k}) |\lesssim \psf^{\vec p}_{\mathcal S_k}(f_{1k},  f_{2k}, f_{3k}). \]
holds true.
We claim that for all $k$
\begin{equation}
\label{Isk}
|I\cap H| \leq 2^{-5} |I| \qquad \forall I \in \mathcal S_k. 
\end{equation}
This is because if \eqref{Isk} fails for $I$, $I$ must be contained in $3Q$ for some $Q\in \mathcal Q$. But the support of $f_{3k}$ is contained in $\widetilde{H}^c$ which does not intersect $3Q$, whence  $\l f_{3k}\r_{I,p_3}=0$.
Relation \eqref{Isk} has the consequence that  if $\{E_I: I\in \mathcal S_k\}$ denote the distinguished pairwise disjoint  subsets of $I \in \mathcal S_k$ with $|E_I| \geq 2^{-2} |I|$, the sets  $\widetilde{E_I}:= E_I \cap H^c$ are also pairwise disjoint and  $|\widetilde{E_I}| \geq 2^{-3} |I|$. By a similar argument to the one used to get to \eqref{toholders} but with $\widetilde{E_I}$ replacing $  E_I$, followed by  H\"older's inequality in $k$ with tuple $\vec r $, and later by H\"older's inequality for the integral with the tuple $\vec t $ from \eqref{rangeVV3}, the left hand side of \eqref{vvpf1} is bounded by
\begin{equation}
\label{toholders2}
 \sum_k \psf^{\vec p}_{\mathcal S_k}(f_{1k},  f_{2k}, f_{3k})\lesssim \int_{H^c}   \left( \prod_{j=1}^3 \left\| \{\mathrm{M}_{p_j} f_{jk}(x)\}\right\|_{\ell^r_j} \right) \,\d x  \lesssim \prod_{j=1}^3 \big\|\{\mathrm{M}_{p_j} f_{jk}\}\big\|_{L^{t_j}( H^c; \ell^{r_j})}
\end{equation}
Now by Fefferman-Stein's inequality since $t_3>p_3$
\begin{equation}
\label{finalh1}
\big\|\{\mathrm{M}_{p_3} f_{3k}\}\big\|_{L^{t_3}( H^c; \ell^{r_j})} \lesssim   \|  \cic{f}_{3} \|_{L^{t_3}(\R; \ell^{r_3})} \leq \|  \cic{1}_{F_3} \|_{t_3} = |F_3|^{\frac{1}{t_3}}
\end{equation}
Further, for $j=1,2$, by log-convexity of $L^{t_j}$-norms
\begin{equation}
\label{finalh2}
\big\|\{\mathrm{M}_{p_j} f_{jk}\}\big\|_{L^{t_j}( H^c; \ell^{r_j})} \leq \big\|\{\mathrm{M}_{p_j} f_{jk}\}\big\|_{L^{q_j}( \R; \ell^{r_j})}^{\frac{q_j}{t_j}} \big\|\{\mathrm{M}_{p_j} f_{jk}\}\big\|_{L^{\infty}( H^c; \ell^{r_j})}^{1-\frac{q_j}{t_j}} \lesssim |F_3|^{\frac{1}{t_j}-\frac{1}{q_j}}  \end{equation}
where, to obtain the final step, we used the Fefferman-Stein inequality to estimate the  $L^{q_j}( \R; \ell^{r_j})$-norm by $O(1)$ and the definition of $H$ to estimate the $L^\infty(H^c; \ell^{r_j})$-norm by $|F_3|^{-\frac{1}{q_j}}.$ Using \eqref{finalh1} and \eqref{finalh2} for $j=1,2$ to bound the right hand side of \eqref{toholders2}
finally yields \eqref{vvpf1} and completes the proof of the Theorem.

 \bibliography{biblioUMD}{}
\bibliographystyle{amsplain}
 
\end{document}